\documentclass[reqno]{amsart}
\usepackage[utf8]{inputenc}
\usepackage{amssymb, amsthm}
\usepackage{mathrsfs}
\usepackage{mathabx}
\usepackage{bm,bbm}
\usepackage{microtype} 
\usepackage{comment} 

\usepackage{amsfonts}
\DeclareMathAlphabet{\mathpgoth}{OT1}{pgoth}{m}{n}
\DeclareMathAlphabet{\mathesstixfrak}{U}{esstixfrak}{m}{n}
\DeclareMathAlphabet{\mathboondoxfrak}{U}{BOONDOX-frak}{m}{n}
\usepackage{amsmath,amssymb}
\usepackage[bbgreekl]{mathbbol}
\usepackage{yfonts,mathtools}

\usepackage{tikz-cd}
\usepackage[all,cmtip]{xy}

\numberwithin{equation}{section}
\usepackage{soul}

\usepackage{hyperref}
\hypersetup{colorlinks}
\definecolor{darkred}{rgb}{0.5,0,0}
\definecolor{darkgreen}{rgb}{0,0.5,0}
\definecolor{darkblue}{rgb}{0,0,0.5}
\hypersetup{colorlinks, linkcolor=darkblue, filecolor=darkgreen, urlcolor=darkred, citecolor=darkblue}
\makeatletter 
\@addtoreset{equation}{section}
\makeatother  

\numberwithin{equation}{section}

\newtheorem{thma}{Theorem}

\newtheorem{thm}{Theorem}[section]
\newtheorem{cor}[thm]{Corollary}

\newtheorem{conj}[thm]{Conjecture}
\newtheorem{prop}[thm]{Proposition}

\newtheorem{lemma}[thm]{Lemma}
\theoremstyle{definition}
\newtheorem{defn}[thm]{Definition}
\theoremstyle{remark}
\newtheorem{rem}[thm]{Remark}
\newtheorem{hyp}[thm]{Hypothesis}

\usepackage{xcolor}
\newtheorem{example}[thm]{Example}

\newcommand{\beq}{\begin{equation}}
\newcommand{\eeq}{\end{equation}}
\newcommand{\beqn}{\begin{equation*}}
\newcommand{\eeqn}{\end{equation*}}
\newcommand{\ov}{\overline}
\newcommand{\mb}{\mathbb}
\newcommand{\mc}{\mathcal}
\newcommand{\mf}{\mathfrak}

\newcommand{\wc}{\widecheck}

\newcommand{\wh}{\widehat}

\newcommand{\eq}{{\it eq}}
\newcommand{\qu}{/\hspace{-0.1cm}/}

\newcommand{\ev}{{\rm ev}}

\title{Quantum Kirwan Map and Quantum Steenrod Operation}

\author{Guangbo Xu}
\address{Department of Mathematics, Rutgers University, Hill Center--Busch Campus, 110 Frelinghuysen Road, Piscataway, NJ 08854-8019, USA}
\email{guangbo.xu@rutgers.edu}

\thanks{The author is supported by NSF DMS-2345030.}

\date{\today}

\begin{document}

\begin{abstract}
We construct an equivariant extension of the quantum Kirwan map and show that it intertwines the classical Steenrod operation on the cohomology of a classifying space with the quantum Steenrod operation of a monotone symplectic reduction. This provides a new method of computing quantum Steenrod operations developed by Seidel--Wilkins. When specialized to the non-equivariant piece, our result also resolves the monotone case of Salamon's quantum Kirwan map conjecture in the symplectic setting. 
\end{abstract} 

\maketitle

\setcounter{tocdepth}{1}

\tableofcontents

\section{Introduction}

Steenrod operations have recently been extended to a quantum version in the context of symplectic topology by taking into account the contributions of holomorphic curves \cite{fukaya-quantization}\cite{wilkins2020construction}\cite{Seidel_Wilkins_2022}. This operation has found  novel applications in many related areas such as Hamiltonian dynamics (see \cite{shelukhin_2020, Shelukhin_2021}, \cite{Cineli_Ginzburg_Gurel_2022}, \cite{Rezchikov_2021}) and arithmetic mirror symmetry \cite{Seidel_2023}. See also \cite{Lee_2023_1, Lee_2023_2} and  \cite{Chen_2024, Chen_2024_Fukaya} for more recent studies on quantum Steenrod operations.

In this paper we establish a formula for the quantum Steenrod operation for symplectic manifolds admitting GLSM (gauged linear sigma model) presentations. Such a formula would facilitate the computation of the quantum Steenrod operation, which is in general a difficult problem. The main ingredient is the quantum Kirwan map originally proposed by Salamon, studied by Ziltener \cite{Ziltener_book} and Woodward \cite{Woodward_15}, and its ${\mb Z}_p$-equivariant extension.

\subsection{Assumptions and notations}

The geometric assumptions are very close to that of Gaio--Salamon \cite{Gaio_Salamon_2005} and Ziltener \cite{Ziltener_book} in the study of the adiabatic limits of the symplectic vortex equation. Let $K$ be a compact connected Lie group. Let $(V, \omega_V)$ be a symplectic manifold with a Hamiltonian $K$-action. Let $\mu: V \to {\mf k}^*$ be a moment map. For any $\eta \in {\mf k}$, the infinitesimal action of $\eta$ is the Hamiltonian vector field ${\mc X}_\eta$ associated to $\mu(\eta)$.

\begin{hyp}\label{hyp1}
{\bf (Regular quotient)} $\mu$ is a proper map and $K$ acts freely on $\mu^{-1}(0)$.    
\end{hyp}

Under this assumption, the symplectic quotient 
\beqn
V \qu K:= \mu^{-1}(0)/K
\eeqn
is a smooth compact manifold with a canonically induced symplectic form.

To guarantee the compactness of vortex moduli, we impose the following convexity condition.

\begin{hyp}\label{hyp2}{\bf (Equivariant convexity)}
There is a $K$-invariant, $\omega_V$-compatible almost complex structure $\wh J_V$, a proper $K$-invariant function $f_V: V \to [0, +\infty)$, and a constant $C>0$ such that
\beqn
f_V(x) \geq C \Longrightarrow \left\{ \begin{array}{l} \langle \nabla_\xi \nabla f_V(x), \xi \rangle + \langle \nabla_{\wh J_V \xi} \nabla f_V(x), \wh J_V \xi \rangle \geq 0\ \forall \xi \in T_x V,\\[0.3cm]
 df_V(x) \wh J_V {\mc X}_{\mu(x)} (x)  \geq 0, \end{array} \right.
\eeqn
\end{hyp}

Indeed, if $V = {\mb C}^N$ and $K$ acts on $V$ via a linear representation $K \to U(N)$, it was shown in \cite{Cieliebak_Gaio_Salamon_2000} that Hypothesis \ref{hyp1} implies Hypothesis \ref{hyp2}.

To reduce the technicality, we make the following assumptions.

\begin{hyp}\label{hyp3}
{\bf (Contractible target)} $V$ is contractible. \end{hyp}

\begin{hyp}\label{hyp4}
{\bf (Equivariant monotonicity)} There is a positive real number $\lambda$ such that 
\beqn
\langle [\omega_V + \mu], d \rangle = \lambda \langle c_1^K(TV), d \rangle \ \forall d \in \pi_2^K(V).
\eeqn
\end{hyp}
As a consequence, the symplectic reduction $V \qu K$ is also aspherically monotone. To simplify the notations, we identify elements of $\pi_2(V \qu K)$ with their images under the natural map $\pi_2(V \qu K) \to \pi_2^K(V)\cong \pi_2(BK)$, where $BK$ is the classifying space of $K$.

For any commutative unital ring $R$, let $\Lambda_R = R((q))$ be the Novikov field of formal Laurent series in a formal variable $q$ with $R$ coefficients, $\Lambda_{R, 0}= R[[q]]$ the Novikov ring, and $\Lambda_{R, +} = q \Lambda_{R, 0}$.

We mainly use cohomology with coefficients either in ${\mb Z}$ or in ${\mb F}_p$. Denote by $H^*(\cdot; {\mb Z})/ {\rm Torsion}$ the free part of the integral cohomology. For quantum cohomology of a symplectic manifold $M$, the notation $QH^*(M; \Lambda_{\mb Z})$ means the quantum cohomology ring with underlying space being 
\beqn
\big( H^*(M; {\mb Z})/{\rm Torsion}\big) \otimes \Lambda_{\mb Z},
\eeqn
the same convention as in \cite{McDuff_Salamon_2004}.

The ${\mb Z}_p$-equivariant cohomology of a point is the following algebra
\beqn
H^*_{{\mb Z}_p}({\rm pt}; {\mb F}_p) = \left\{ \begin{array}{ll} {\mb F}_2[t, \theta]/ \langle \theta^2 - t \rangle,\ &\ p = 2,\\
{\mb F}_p[t, \theta]/ \langle t \theta - \theta t,\ \theta^2 \rangle,\ &\ p > 2 \end{array}\right.
\eeqn
Throughout this paper, the variables $t$ and $\theta$ always represent the variables in this equivariant cohomology satisfying such relations. In particular, for any algebra $R$ over ${\mb F}_p$, we use $R[t, \theta]$ to denote the algebra generated by $R$ and $t, \theta$ satisfying the above relations. 

\subsection{Quantum Kirwan map}

For a general symplectic reduction $V\qu K$, the Kirwan map is the compositions of the two natural maps 
\beqn
\xymatrix{ H_K^*(V) \ar[r] & H_K^*(\mu^{-1}(0)) \ar[r] & H^*( \mu^{-1}(0)/K) = H^*( V \qu K)}
\eeqn
which respects the multiplicative structure of cohomology. Here the first map is induced from the $K$-equivariant inclusion $\mu^{-1}(0) \hookrightarrow V$ and the second map is an isomorphism when $K$ acts freely on $\mu^{-1}(0)$. The ``quantum version'' of the Kirwan map was proposed by Salamon following the works of Dostoglou--Salamon \cite{Dostoglou_Salamon} and Gaio--Salamon \cite{Gaio_Salamon_2005}. The first result of this paper is a proof of this conjecture under the monotonicity assumption.

\begin{thma}\label{thma}
Under Hypothesis \ref{hyp1}---\ref{hyp4}, there is a ${\mb Z}$-linear map 
\beqn
\kappa: H_K^*(V; {\mb Z}) \to QH^*(V\qu K; \Lambda_{{\mb Z}})
\eeqn
such that 
\beq\label{qk_relation}
\kappa (a) \ast \kappa (b) = \kappa (a\cup b),\ \forall a, b \in H_K^*(V; {\mb Z}).
\eeq
\end{thma}

We give an intuitive description of the quantum Kirwan map which was originally due to Salamon. We need to consider {\bf affine vortices}. These are solutions to the symplectic vortex equation over the complex plane ${\mb C}$. Roughly, an affine vortex is a map $u: {\mb C} \to V$ which is holomorphic after twisting by a gauge field $A = d + \phi ds + \psi dt$. More precisely, the pair $(A, u)$ needs to satisfy 
\begin{align*}
&\ \partial_s u + {\mc X}_\phi (u) + \wh J (\partial_t u + {\mc X}_\psi(u)) = 0,\ &\ \partial_s \psi - \partial_t \phi + [\phi, \psi] + \mu(u) = 0.
\end{align*}
Modulo gauge symmetry, these affine vortices form finite-dimensional moduli spaces indexed by the degree $d \in H_2^K(V; {\mb Z})$. Let ${\mc M}_1^{\rm vortex}(d)$ denote temporarily the moduli space of affine vortices of degree $d$ with one interior marking. Then there are two evaluation maps
\begin{align}
&\ \ev_{d, 0}: {\mc M}_1^{\rm vortex}(d) \to V_K,\ &\ \ev_{d,\infty}: {\mc M}_1^{\rm vortex}(d) \to V\qu K.
\end{align}
Here $V_K$ is the Borel construction of $V$. Then formally one can define 
\beqn
\kappa (b) = \sum_{d} q^d (\ev_{d, \infty})_* (\ev_{d, 0}^*(b) ).
\eeqn
The relation \eqref{qk_relation} follows from the description of 1-dimensional moduli spaces of affine vortices with two marked points.

\subsection{The quantum Steenrod operation and the equivariant quantum Kirwan map}

In ${\mb F}_p$-coefficients, the Steenrod operations are a collection of linear maps
\beqn
\Sigma_b: H^*(X; {\mb F}_p) \to H^*(X; {\mb F}_p)
\eeqn
labelled by $b \in H^*(X; {\mb F}_p)$. When $X$ is a manifold, the Steenrod operation can be defined via the Morse model (see \cite{Betz_Cohen_1994}). Indeed, the cup product (or intersection product) can be defined by counting (perturbed) flow trees with two incoming edges and one outgoing edge. In a similar manner, Steenrod operations can be constructed by counting flow trees with $p$ incoming edges and one outgoing edge, while the counts need to be taken equivariantly with respect to the cyclic shuffling of the $p$ incoming edges.

The idea of quantum Steenrod operation was due to Fukaya \cite{fukaya-quantization}. It deforms the  classical one by inserting holomorphic spheres in the center of the flow tree. This idea was firstly rigorously carried out by Seidel \cite{Seidel_pants}, and then by Wilkins \cite{wilkins2020construction} and Seidel--Wilkins \cite{Seidel_Wilkins_2022}. For any monotone symplectic manifold $(M, \omega)$, we denote the operation by 
\beqn
Q\Sigma_b: QH^*(M; \Lambda_{{\mb F}_p}) \to QH^*(M; \Lambda_{{\mb F}_p}) [t, \theta].
\eeqn


By observing the domain symmetry of the vortex equation, one naturally expects that the quantum Kirwan map admits a ${\mb Z}_p$-equivariant extension. Indeed, the rotational symmetry of the vortex equation is not used in the definition of $\kappa$. By requiring the perturbation data on the domain ${\mb C}$ to satisfy a ${\mb Z}_p$-equivariance condition, in a way similar to the geometric construction of the quantum Steenrod operation, one can define $\kappa^{eq}$ via  certain equivariant counts of affine vortices. Furthermore, by imitating the proof of Theorem \ref{thma} and that of \cite[Proposition 4.8]{Seidel_Wilkins_2022}, one obtains an equivariant analogue of the quantum Kirwan map conjecture. This is the second main result of this paper which is stated here. 

\begin{thma}\label{thmb}
Under Hypothesis \ref{hyp1}---\ref{hyp4}, there is an ${\mb F}_p$-linear map 
\beqn
\kappa^{\eq}: H^*(BK; {\mb F}_p) \to QH^*(V \qu K; \Lambda_{{\mb F}_p})[t, \theta]
\eeqn
(called the {\bf ${\mb Z}_p$-equivariant quantum Kirwan map}) satisfying 
\begin{align*}
&\ \kappa^{eq}(1) = 1,\ &\ \kappa^{eq}|_{q = 0} = \kappa|_{q= 0},
\end{align*}
and for all $a, b \in H^*(BK; {\mb F}_p)$, 
\beqn
Q\Sigma_{\kappa(b)}( \kappa^{eq}(a)) = \kappa^\eq (\Sigma_b(a)).
\eeqn
\end{thma}


\begin{rem}
There should be a parallel picture for Lagrangian correspondences. Under suitable monotonicity condition, a Lagrangian correspondence $M_1 \overset{L}{\longrightarrow} M_2$ induces a map $QH^*(M_1) \to QH^*(M_2)$ by counting pseudoholomorphic quilts. By considering the ${\mb Z}_p$-equivariant version, one should obtain a map in characteristic $p$ which intertwines with the quantum Steenrod operations. 
\end{rem}

\begin{rem}
Classical Steenrod operation on $H^*(BK)$ can be calculated by algebraic topological method (see for example \cite{Borel_Serre_1953}). Therefore Theorem \ref{thmb} potentially provides a new way of computing the quantum Steenrod operations on certain GIT quotient in the same spirit as the GLSM computation of quantum cohomology. To do this, one needs to be able to compute the equivariant quantum Kirwan map. The non-equivariant case has been carried out in various cases, see for example \cite{Gonzalez_Woodward_2019}, via explicit identification of the affine vortex moduli (see \cite{VW_affine} \cite{Guangbo_vortex}). To compute $\kappa^{eq}$, one needs to understand the ${\mb Z}_p$-action on the moduli spaces and apply the ${\mb Z}_p$-version of fixed point localization. 
\end{rem}

\begin{rem}
There could be another approach of computing the quantum Steenrod operations for toric manifolds using Seidel representation \cite{Seidel_representation}, in the same spirit as computing the quantum cohomology (see \cite{Tseng_Wang_2012}) using Seidel representation.
\end{rem}

\begin{rem}
There is an interesting distinction between the quantum Steenrod operation and the equivariant quantum Kirwan map. In characteristic $p$, in low degrees one expects that $Q\Sigma$ is determined by the classical Steenrod operation and quantum cohomology as the nontrivial quantum equivariant effect is related to certain $p$-fold multiple covers of holomorphic spheres which are ${\mb Z}_p$-fixed points of the moduli space. In contrast, the moduli space of affine vortices, even in low degrees (for example, degree $1$ for $V \qu K = \mb{CP}^1$), have ${\mb Z}_p$-fixed points, as the domain ${\mb Z}_p$-symmetry may be absorbed by gauge symmetry (essentially the target symmetry). 
\end{rem}

There is a simple algebraic consequence of Theorem \ref{thmb}. A subspace $I \subset H^*(X; {\mb F}_p)[t, \theta]$ is called a {\bf Steenrod ideal} if
\beqn
\Sigma_b(I) \subseteq I\ \forall b \in H^*(X; {\mb F}_p).
\eeqn
This is an equivariant analogue of the notion of (quantum) Stanley--Reisner ideal (see discussions in \cite{Gonzalez_Woodward_2019} for the toric case).

\begin{cor}
The kernel of $\kappa^{eq}: H^*(BK;\Lambda_{{\mb F}_p} )[t, \theta] \to H^*(V \qu K; \Lambda_{{\mb F}_p})[t, \theta]$ is a (nontrivial) Steenrod ideal.
\end{cor}

\subsection{Conjectures about general situation}

If we drop Hypothesis \ref{hyp3} and Hypothesis \ref{hyp4}, our Theorem \ref{thma} and Theorem \ref{thmb} should still hold with suitable modifications. The counterpart of  Theorem \ref{thma} is just the general version of the quantum Kirwan map conjecture, stated in precise terms in \cite{Ziltener_book}.

\begin{conj}[Quantum Kirwan map conjecture]\label{conj18}
Counting affine vortices defines a morphism of cohomological field theories from the equivariant Gromov--Witten theory of $V$ to the Gromov--Witten theory of $V \qu K$. In particular, in the case of quantum cohomology, there is a linear map
\beqn
\kappa: QH_K^*(V) \to QH^*(V \qu K)
\eeqn
and a quantum cohomology class
\beq\label{curvature}
\tau \in QH^{\rm even} (V \qu K)
\eeq
such that for all $x, y \in QH_K^*(V)$ there holds
\beqn
\kappa( x \ast_K y ) = \kappa(x) \ast_{\tau} \kappa (y)
\eeqn
where $\ast_K$ is the equivariant quantum cup product of $V$ (which is the classical cup product if $V$ is aspherical) and $\ast_\tau$ is the quantum cup product of $V \qu K$ at the bulk class $\tau$. 
\end{conj}

Without the monotonicity assumption, the moduli spaces cannot be regularized by geometric perturbations. Hence to prove Conjecture \ref{conj18} certain virtual technique is necessary (compare with \cite{Woodward_15}). Moreover, orbifold points will affect the count and the above conjecture only hold in rational coefficients; in fact, beyond the semipositive case the quantum cohomology is only defined over rational numbers. 

On the other hand, recently Bai and the author developed the idea of Fukaya--Ono \cite{Fukaya_Ono_integer} and defined integer-valued curve counting invariants (see \cite{Bai_Xu_2022, Bai_Xu_Arnold}), using the stable complex structures on the moduli spaces, even beyond the semipositive case. Therefore, one can construct the integer version of the quantum cohomology as well as the integer version of the quantum Kirwan map, denoted by $\kappa^{\mb Z}$. This construction should also allow one to define the quantum Steenrod operations beyond the semipositive case. Moreover, to include the term $\tau\in QH^{\rm even}(V \qu K)$ of \eqref{curvature}, one needs to define a ``bulk deformation'' of the quantum Kirwan map, which is denoted by 
\beqn
Q\Sigma_{b, \tau}: QH^*(V \qu K; {\mb F}_p) \to QH^*(V \qu K; {\mb F}_p)[t, \theta].
\eeqn
With all these understood, the generalization of Theorem \ref{thmb} can be stated as follows.

\begin{conj}
There exists an ${\mb F}_p$-version of the quantum Kirwan map 
\beqn
\kappa: QH_K^*(V; \Lambda_{{\mb F}_p}) \to H^*(V \qu K; \Lambda_{{\mb F}_p}),
\eeqn
an equivariant quantum Kirwan map
\beqn
\kappa^{\it eq}: QH_K^*(V; \Lambda_{{\mb F}_p}) \to H^*(V \qu K; \Lambda_{{\mb F}_p})[t, \theta],
\eeqn
and a cohomology class
\beqn
\tau \in QH^{\rm even} (V \qu K; \Lambda_{{\mb F}_p})
\eeqn
such that for all $a \in QH_K^*(V;\Lambda_{{\mb F}_p})$, one has 
\beqn
\kappa^{eq}  ( Q\Sigma_b^K(a)) = Q\Sigma_{\kappa(b), \tau} (\kappa^{eq}(a)).
\eeqn
\end{conj}

\subsection{Acknowledgements}

The author thanks Jae Hee Lee and Shaoyun Bai for stimulating discussions, and Mark Grant for discussions on mathoverflow. 

This work is partially supported by NSF Grant No. DMS-2345030.



\section{Affine vortices and quantum Kirwan map}\label{section2}

\subsection{Affine vortices}

\begin{defn}
An {\bf affine vortex} (with respect to a $K$-invariant $\omega_V$-compatible almost complex structure $\wh J$) is a pair $(A, u)$ where $A = d + \phi ds + \psi dt$ is a connection on the trivial $K$-bundle $P = {\mb C} \times K$ over the complex plane ${\mb C}$, $u: {\mb C} \to V$ is a smooth map such that
\begin{enumerate}
    \item The pair $(A, u)$ satisfies the vortex equation 
    \begin{align}\label{vortex_eqn}
    &\ \partial_s u + {\mc X}_\phi (u) + \wh J (\partial_t u + {\mc X}_\psi(u)) = 0,\ &\ F_A + \mu(u) ds dt = 0.
    \end{align}

    \item The energy of $(A, u)$, defined by 
    \beq
    E(A, u) = \| \partial_s u + {\mc X}_\phi(u)\|_{L^2}^2 + \| \mu(u)\|_{L^2}^2,
    \eeq
    is finite.
\end{enumerate}
\end{defn}
The group of gauge transformations is
\beqn
{\mc G}:= \{ g: {\mb C} \to K \},
\eeqn
which acts on the set of pairs $(A, u)$ and the (left) action is denoted by 
\beqn
g\cdot (A, u) = ( d -g^{-1} dg + {\rm Ad}_g (\phi) ds + {\rm Ad}_g (\psi) dt, g u).
\eeqn
The vortex equation \eqref{vortex_eqn} is invariant under gauge transformation.

It is proved by Ziltener \cite{Ziltener_Decay} that any affine vortex ``closes up'' at infinity. More precisely, for any affine vortex $(A, u)$, the continuous map $u/K: {\mb C} \to V/K$, which is independent of gauge, has a well-defined limit at infinity which is contained in $\mu^{-1}(0)/K$. Hence there is a gauge-invariant evaluation
\beqn
\ev_\infty(A, u)\in V \qu K.
\eeqn
Moreover, there exists a $K$-bundle $P \to S^2$ and a section $\hat u: S^2 \to P(V)$, which agrees with $u$ on ${\mb C} = S^2 \setminus \{\infty\}$ with respect to a suitable trivialization of $P$ away from $\infty$. On the other hand, any continuous section $\hat u: S^2 \to P(V)$ represents a spherical equivariant class $d \in \pi_2^K(V)$. 

The affine vortex equation \eqref{vortex_eqn} is invariant under translations of the domain ${\mb C}$. We usually consider moduli spaces of solutions modulo both gauge transformation and translation. For any $k \geq 0$, let 
\beqn
{\mc M}_k^{\rm vortex}(d)
\eeqn
be the moduli space of $k$ marked affine vortices with degree $d$, whose virtual dimension is 
\beqn
{\rm dim}^{\rm vir} {\mc M}_k^{\rm vortex}(d) = {\rm dim} V \qu K + 2c_1(d) + 2k - 2. 
\eeqn

\begin{prop}
Choose a class $d \in \pi_2^K(V)$.
\begin{enumerate}
\item (\cite{Ziltener_book}) The energy of any affine vortex $(A, u)$ representing $d$ is 
\beqn
E(A, u) = \langle [\omega_V + \mu], d \rangle.
\eeqn

\item (\cite{Venugopalan_Xu}) There is a Banach manifold ${\mc B}_d$, a Banach vector bundle ${\mc E}_d \to {\mc B}_d$ and a Fredholm section ${\mc F}_d: {\mc B}_d \to {\mc E}_d$ 
such that ${\mc F}^{-1}(0)$ is homeomorphic to ${\mc M}_1^{\rm vortex}(d)$. Moreover, the index of ${\mc F}_d$ is equal to 
\beqn
{\rm dim}^{\rm vir} {\mc M}_1^{\rm vortex}(d) = {\rm dim} V \qu G + 2 c_1(d).
\eeqn

\item (\cite{Venugopalan_Xu}) When ${\mc F}_d$ is transverse, ${\mc M}_1^{\rm vortex}(d)$ has the structure of a smooth manifold and the evaluation map 
\beqn
\ev_\infty: {\mc M}_1^{\rm vortex}(d) \to V \qu K
\eeqn
is a smooth map.
\end{enumerate}
\end{prop}

\subsection{Compactifications}

\subsubsection{Ziltener compactification}

Ziltener \cite{Ziltener_book} provided a natural compactification of the moduli space of affine vortices. In general, given an energy bound, a sequence of solutions of the vortex equation can have energy concentrations at isolated points, causing sphere bubbles in $V$; this is excluded in our setting by assuming $V$ is contractible. There are other two types of noncompactness behaviors which can be easily described in terms of the energy distribution. First, the energy distribution could separate in different regions of the domain ${\mb C}$ which are infinitely far away, causing affine vortices to ``split.'' Second, the energy distribution could spread out over larger and larger regions; equivalently, energy escape at the infinity of ${\mb C}$ and forming sphere bubbles in the symplectic quotient $V \qu K$. To incorporate both cases, one can introduce the notion of ``stable affine vortices.'' 

We recall the combinatorial description of domain curves of stable affine vortices. A {\bf scaled tree} is a rooted tree $\Gamma$ with a set of vertices $V_\Gamma$ (corresponding to irreducible components of a nodal curve), a set of edges (corresponding to nodes), and a set of leaves $L_\Gamma$ (corresponding to interior marked points). We also assume that the root is implicitly attached with an ``outgoing'' leaf. W The vertices are ordered by the root: $v > v'$ if $v'$ is closer to the root than $v$, and $v \succ v'$ if $v > v'$ and they are adjacent. Moreover, the structure of scaled tree also contains a functor 
\beqn
\mf{scale}: (V_\Gamma, \succ) \to \{ 0 \to 1 \to \infty\}
\eeqn
satisfying the following conditions.
\begin{itemize}
\item If $\mf{scale}(v) \leq 1$, $\mf{scale}(v') \geq 1$, then along the path connecting $v$ and $v'$ there is exactly one vertex in $\mf{scale}^{-1}(1)$.

\item Leaves are attached to vertices in $\mf{scale}^{-1}(\{0, 1\})$.
\end{itemize}
A scaled tree is {\bf stable} if each $v \in \mf{scale}^{-1}(\{0,\infty\})$ has valence (both edges and leaves counted) is at leas three, and each $v' \in \mf{scale}^{-1}(1)$ has valence at least two. 

Given a scaled tree $\Gamma$, a {\bf scaled curve} of type $\Gamma$ is the union $\Sigma$ of copies of $S^2 \cong \Sigma_v$ indexed by vertices $v \in V_\Gamma$, together with marked points corresponding to (incoming) leaves attached to corresponding components. The edge connecting $v \succ v'$ corresponds to $\infty \in \Sigma_v$ and a finite point of $\Sigma_{v'}$. Two scaled curves are isomorphic if there are componentwise biholomorphisms, with the restriction that the biholomorphic map on a component corresponding to $v \in \mf{scale}^{-1}(1)$ has to be a translation of ${\mb C} = S^2\setminus \{\infty\}$. Denote by $\Sigma^{(0)}, \Sigma^{(1)}, \Sigma^{(\infty)}$ be the union of components corresponding to vertices of scale $0$, $1$, and $\infty$ respectively.

For each $k \geq 1$, let $\ov{\mc M}{}_k^{\rm scaled}$ be the moduli space of stable scaled curves with $k$ interior marked points. The specific way of defining isomorphisms result in many differences from the Deligne--Mumford space $\ov{\mc M}_k$, for example, in dimension. In fact, the top stratum of $\ov{\mc M}{}_k^{\rm scaled}$ can be identified as the moduli space of $k$ distinct points of ${\mb C}$ modulo translation, hence has dimension $2k - 2$. 

A stable affine vortex has a domain $\Sigma$ being a scaled curve of certain combinatorial type $\Gamma$. On each component $\Sigma_v \subset \Sigma^{(0)}$, it is a $\wh J$-holomorphic sphere $u_v: \Sigma_v \to V$; on each component $\Sigma_v \subset \Sigma^{(1)}$, it is an affine vortex $(A_v, u_v)$ with domain $\Sigma_v$; on each component $\Sigma_v \subset \Sigma^{(\infty)}$, it is a $J$-holomorphic sphere $u_v: \Sigma_\gamma \to V \qu K$, where $J$ is the induced almost complex structure on $V \qu K$. A matching condition is required for all nodes; in particular, for $v \in \mf{scale}^{-1}(1)$ adjacent to $v' \in \mf{scale}^{-1}(\infty)$ with a node having coordinate $w_{vv'}\in \Sigma_{v'}$, one needs to require
\beqn
\ev_\infty( A_v, u_v) = u_{v'} (w_{vv'}) \in V \qu K.
\eeqn

Ziltener \cite{Ziltener_book} showed that the collection of all equivalence classes of stable affine vortices of degree $d$ with $k$ marked points, denoted by $\ov{\mc M}{}_k^{\rm vortex}(d)$, is a natural compactification of the moduli space ${\mc M}{}_k^{\rm vortex}(d)$. 

One needs to allow the almost complex structure to depend on the moduli parameter in $\ov{\mc M}{}_k^{\rm scaled}$ or other parameters. Let $\ov{\mc C}{}_k^{\rm scaled} \to \ov{\mc M}{}_k^{\rm scaled}$ be the universal curve, which is a smooth manifold away from nodes. Suppose $\wh J$ is a family of domain-dependent almost complex structures on $V$ depending on points in $\ov{\mc C}{}_k^{\rm scaled}$ which is a constant near nodes. Then one can consider stable affine vortices defined by $\wh J$; on unstable components the equations are for certain constant almost complex structure.

\subsubsection{Cusp compactification}

In order to use pseudocycles inside the Borel construction, we introduce a different compactification called the {\bf cusp compactification}. It is similar to the compactification by cusp curves used by Gromov \cite{Gromov_1985}, where we replace a multiply-covered component by its underlying simple curve.

\begin{defn}
\begin{enumerate}

\item Let $\Gamma$ be a scaled tree (which is not necessarily stable). Its {\bf trimming}, denoted by $\Gamma^\vee$, is the scaled tree obtained by removing all vertices and edges of $\Gamma$ which do not belong to any maximal path connecting a leaf and the root.

\item The trimming of a scaled curve $(\Sigma, {\bf z})$ with underlying scaled tree $\Gamma$, denoted by $(\Sigma^\vee, {\bf z}^\vee)$, is the scaled curve obtained by removing all components $\Sigma_v$ with $v \notin V_{\Gamma^\vee}$. 

\item Let $(A, u, {\bf z})$ be a marked stable affine vortex over the curve $(\Sigma, {\bf z})$. Its {\bf trimming} is the stable affine vortex denoted by $(A^\vee, u^\vee, {\bf z}^\vee)$, whose domain is $(\Sigma^\vee, {\bf z}^\vee)$ and whose components are defined as follows: if $\Sigma_v \subset \Sigma^\vee$ is stable, then $(A_v^\vee, u_v^\vee) = (A_v, u_v)$; if $\Sigma_v \subset \Sigma^\vee$ is unstable (which is necessarily a nonconstant holomorphic sphere in $V \qu K$), then $u_v^\vee: \Sigma_v^\vee \to V \qu K$ is the underlying simple curve.
\end{enumerate}
\end{defn}

The trimming of stable affine vortices defines an equivalence relation on the moduli space. For each $k$, denote 
\beqn
\widecheck {\mc M}_k^{\rm vortex}(d) := \ov{\mc M}{}_k^{\rm vortex}(d)/ {\rm trimming}
\eeqn
equipped with the quotient topology. We call this quotient the {\bf cusp compactification} of the moduli space of affine vortices (with $k$ marked points).  Notice that any element of $\wc {\mc M}_k^{\rm vortex}(d)$ is also an equivalence class of stable affine vortices of possibly different degrees. The following statement is easy to verify.

\begin{lemma}
The cusp compactification $\wc {\mc M}{}_k^{\rm vortex}(d)$ is compact and Hausdorff.
\end{lemma}

The cusp compactification is also stratified by the combinatorial types. Each stratum, denoted by $\alpha$, is a combinatorial type of a stable affine vortex of degree $d_\alpha$ with $c_1(d_\alpha) \leq c_1(d)$. 

\begin{lemma}
For generic domain-dependent almost complex structure $\wh J$, each stratum $\partial^\alpha \wc{\mc M}{}_k^{\rm vortex}(d)$ is a smooth manifold. Moreover, under the monotonicity assumption, for each boundary stratum $\alpha$, one has
\beqn
{\rm dim} \partial^\alpha \wc{\mc M}{}_k^{\rm vortex}(d) \leq {\rm dim}^{\rm vir} \wc{\mc M}{}_k^{\rm vortex}(d) - 2
\eeqn
\end{lemma}

\begin{proof}
Constant almost complex structures on $V \qu K$ can make simple holomorphic stable spheres transverse. Moreover, all kinds of degeneration  increase the codimension by at least two.
\end{proof}

\subsection{The Poincar\'e bundles}

The purpose here is to define natural principal bundle, one for each marked point, over the moduli space of affine vortices. We first explain the definition when the underlying marked scaled curve $(\Sigma, {\bf z})$ is smooth and fixed. Here $\Sigma \cong {\mb C}$ and ${\bf z}$ is a list of $k$ marked points. Let $\tilde {\mc M}_{(\Sigma, {\bf z})}^{\rm vortex}(d)$ be the set of solutions $(A, u)$ to the affine vortex equation on $(\Sigma, {\bf z})$. Let 
\beqn
{\mc G}_{(\Sigma, {\bf z})}
\eeqn
be the group of smooth gauge transformations on the domain $\Sigma = {\mb C}$; for each $j$, let 
\beqn
{\mc G}_{z_j} \subset {\mc G}_{(\Sigma, {\bf z})}
\eeqn
be the subgroup of gauge transformations whose value at $z_j$ is the identity of $K$. Then 
\beqn
{\mc G}/ {\mc G}_{z_j} \cong K
\eeqn
which can be identified with the group of constant gauge transformations. Then the quotient
\beqn
\tilde {\mc M}_{(\Sigma, {\bf z})}^{\rm vortex}(d)/ {\mc G}_{z_j}
\eeqn
is a $K$-bundle over the moduli space 
\beqn
{\mc M}_{(\Sigma, {\bf z})}^{\rm vortex} (d) = \tilde {\mc M}_{(\Sigma, {\bf z})}^{\rm vortex}(d)/ {\mc G}_{(\Sigma, {\bf z})}
\eeqn
because the action by constant gauge transformations is free.

The above notion extends to the case when $(\Sigma, {\bf z})$ is a general scaled marked curve (not necessarily stable). If the underlying scaled tree is $\Gamma$, then ${\mc G}_{(\Sigma, {\bf z})}$ is the group of gauge transformations on 
\beqn
\Sigma^{(1)}:= \bigcup_{\mf{scale}(v) = 1} \Sigma_v.
\eeqn
Then for each marked point $z_j$, one has a $K$-bundle constructed similarly, called the Poincar\'e bundle over the moduli space ${\mc M}_{(\Sigma, {\bf z})}^{\rm vortex}(d)$. 

One can allow $(\Sigma, {\bf z})$ (or its stabilization) to vary in a given stratum of the domain moduli and one obtains a Poincar\'e bundle over the corresponding stratum of the vortex moduli. For each such stratum $\alpha$ and $j = 1, \ldots, k$, denote by 
\beqn
\partial^\alpha \ov{\mc P}_j \to \partial_\alpha \ov{\mc M}{}_k^{\rm vortex}(d)
\eeqn
the obtained Poincar\'e bundle. Denote by 
\beqn
\ov{\mc P}_j:= \bigsqcup_\alpha \partial^\alpha \ov{\mc P}_j \to \ov{\mc M}{}_k^{\rm vortex}(d).
\eeqn
The same construction can be carried over to the cusp compactification because the cusp compactification contains configurations which only collapse holomorphic spheres downstairs or vortex components which do not have marked points. Denote the corresponding Poincar\'e bundles to be 
\beqn
\wc {\mc P}_j \to \wc {\mc M}{}_k^{\rm vortex}(d).
\eeqn

\begin{example}
Consider the moduli space ${\mc M}_1^{\rm vortex}(d_0)$ with $d_0 = 0 \in \pi_2^K(V)$, i.e., constant vortices. It is compact itself and is homeomorphic to $V \qu K$. The Poincar\'e bundle ${\mc P}_1 \to {\mc M}_1^{\rm vortex}(d_0)$ is the bundle $\mu^{-1}(0) \to V \qu K$ which is generally nontrivial. 
\end{example}

\begin{lemma}
The Poincar\'e bundle $\wc {\mc P}_j \to \wc{\mc M}{}_k^{\rm vortex}(d)$ admits a classifying map into a finite-dimensional approximation $B_n K$ of the classifying space $BK$.
\end{lemma}

\begin{proof}
The usual statement for the existence of classifying maps assumes the base of the principal bundle to be a CW complex, which is not necessarily the case for $\wc {\mc M}:= \wc{\mc M}{}_k^{\rm vortex}(d)$. However, the classifying map actually exists for numerable bundles (see \cite[Section 7]{Dold_1963}). As the moduli spaces are compact, all principal bundles are numerable hence the classifying maps do exist.

Moreover, the moduli space $\wc{\mc M}$ is stratified by finitely many smooth manifolds. Inductively, suppose we can perturb a classifying map so that its restriction to all boundary strata of a stratum $\partial^\alpha \wc {\mc M}$, denoted by $\partial( \partial^\alpha \wc {\mc M})$, takes value in $B_n K$. As a neighborhood of $B_n K$ in $BK$ retracts to $B_n K$, one can homotopy the classifying map so that over a neighborhood of $\partial(\partial^\alpha \wc{\mc M})$ the value is contained in $B_n K$. On the other hand, as the interior of $\partial^\alpha \wc {\mc M}$ is a manifold, one can choose the classifying map so that it stays in $B_{n'} K$ for a possibly larger $n' \geq n$. As there are only finitely many strata, the claim follows. 
\end{proof}





\subsection{The quantum Kirwan map via Morse model}

\begin{lemma}
Fix the Morse--Smale pair $(f_\infty, g_\infty)$ on $V \qu K$. There exists a domain-dependent almost complex structure $\wh J$ such that for all $d\in \pi_2^K(X)$, the restriction of $\ev_\infty$ to each stratum $\partial^\alpha \wc{\mc M}{}_1^{\rm vortex}(d)$ is transverse to all stable manifolds of $f_\infty$. 
\end{lemma}

Fix a $\wh J$ satisfying the above lemma. For each sufficiently large $n$, choose a Morse--Smale pair $(\wh f_0, \wh g_0)$ on $B_n K$. For any classifying map 
\beqn
\wh\rho_0: \wc{\mc M}{}_1^{\rm vortex}(d) \to B_n K,
\eeqn
define
\beqn
\begin{split}
\wc{\mc M}{}_1^{\rm vortex}(d; \wh x_0, x_\infty) = &\ \wc {\mc M}{}_1^{\rm vortex}(d) \cap \wh\rho_0^{-1}( W^u(\wh x_0)) \cap \ev_\infty^{-1}(W^s(x_\infty)) \\
\cong &\ \left\{(u, \wh y_0, y_\infty)\ \left|\begin{array}{c} u \in \wc{\mc M}{}_1^{\rm vortex}(d),\  \wh y_0: I_0 \to B_n K,\ y_\infty: I_\infty \to V \qu K,\\
\wh y_0' +\nabla \wh f_0(\wh y_0) = 0,\ \wh y_0(\infty) = \wh x_0,\ \wh y_0(0) = \wh \rho_0(u),\\
y_\infty' + \nabla f_\infty(y_\infty) = 0,\ y_\infty(+\infty) = x_\infty,\ \ev_\infty(u) = y_\infty(0).
\end{array} \right. \right\}
\end{split}
\eeqn
Here $I_0 = (-\infty, 0]$ and $I_\infty = [0, +\infty)$. We also allow $\wh y_0$ resp. $y_\infty$ to have breakings. The stratification of the above moduli space is indexed by $\alpha$ containing the combinatorial type of the affine vortex and the breakings of the trajectories.

We look at moduli spaces of virtual dimension at most one. It is easy to see that the expected dimension of $\wc{\mc M}{}_1^{\rm vortex}(d; \wh x_0, x_\infty)$ is 
\beqn
{\rm dim}^{\rm vir} \wc{\mc M}{}_1^{\rm vortex}(d; \wh x_0, x_\infty) = 2c_1(d) - |\wh x_0| + |x_\infty|
\eeqn
where $|\wh x_0|$, $|x_\infty|$ are the cohomological degrees, i.e., the complements of Morse indices.

\begin{lemma}\label{lemma29}
When $2c_1(d) - |\wh x_0| + |x_\infty| \leq 1$, one can perturb the classifying map so that the following conditions are satisfied.
\begin{enumerate}
\item When the virtual dimension is negative, $\wc {\mc M}{}_1^{\rm vortex}(d; \wh x_0, x_\infty)$ is empty.

\item When the virtual dimension is zero, $\wc {\mc M}{}_1^{\rm vortex}(d; \wh x_0, x_\infty)$ consists of finitely many points which do not contain broken trajectories or singular affine vortices.

\item When the virtual dimension is one, $\wc {\mc M}{}_1^{\rm vortex}(d; \wh x_0, x_\infty)$ is a compact 1-dimensional manifold with boundary, with boundary points corresponding to once-broken configurations without singular affine vortices.
\end{enumerate}
\end{lemma}

\begin{proof}
Fix $d\in \pi_2^K(V)$. We perturb the classifying map $\wh \rho_0$ on $\wc{\mc M}{}_1^{\rm vortex}(d)$ inductively on its strata. Let $\partial^\gamma \wc{\mc M}{}_1^{\rm vortex}(d)$ be a lowest stratum. It is a smooth manifold itself. If it is not the top stratum, then for dimensional reason, for any $\wh x_0$ and $x_\infty$ satisfying $2c_1(d) - |\wh x_0| + |x_\infty| \leq 1$, one can perturb the classifying map to a smooth map so that $\partial^\gamma \wc{\mc M}{}_1^{\rm vortex}(d) \cap \wh \rho_0^{-1}(W^u(\wh x_0)) \cap \ev_\infty^{-1}(W^s(x_\infty)) = \emptyset$. If $\gamma$ is the top stratum of $\wc{\mc M}{}_1^{\rm vortex}(d)$, then the classifying map can be chosen to be a smooth map satisfying the transversality condition. 

Now suppose we have constructed a classifying map $\wh \rho_0$ on a closed union of strata $\partial^\gamma \wc{\mc M}{}_1^{\rm vortex}(d)$ 
such that $\partial^\gamma \wc{\mc M}{}_1^d(d) \cap \wh \rho_0^{-1}(W^s (\wh x_0)) \cap \ev_\infty^{-1}( W^s(x_\infty)) = \emptyset$. Suppose these strata contain all boundary of a stratum $\partial^\alpha \wc{\mc M}{}_1^{\rm vortex}(d)$ over which we would like to extend the classifying map. Start from an arbitrary continuous extension of $\wh \rho_0$ to this stratum. As $\partial^\alpha \wc{\mc M}{}_1^{\rm vortex}(d)$ is itself a smooth manifold, a small smooth perturbation can be made so that $\partial^\alpha \wc{\mc M}{}_1^{\rm vortex}(d) \cap \wh \rho_0^{-1}( W^u(\wh x_0)) \cap \ev_\infty^{-1}( W^s(x_\infty))$ is transverse. Moreover, this intersection is nonempty if and only if $\alpha$ is a top stratum. Lastly, the statement about the index one case follows from the basic gluing construction of Morse trajectories. 
\end{proof}

Choosing orientations on the unstable manifolds $W^u(\wh x_0)$ and $W^u(x_\infty)$. The natural almost complex orientation on the vortex moduli induces a count 
\beqn
m_{n, d}(\wh x_0, x_\infty) = \left\{ \begin{array}{cc} \# \wc {\mc M}{}_1^{\rm vortex}(d; \wh x_0, x_\infty) \in {\mb Z},\ &\ 2c_1(d) - |\wh x_0| + |x_\infty| = 0,\\
0,\ &\ {\rm otherwise} \end{array} \right.
\eeqn

It follows that one can define chain maps
\beqn
\kappa_{n,d}: CM^*(\wh f_0) \to CM^{* - 2c_1(d)} (f_\infty)
\eeqn
(for any coefficient ring). It is easy to show that up to homotopy $\kappa_{n, d}$ is independent of the choice of the almost complex structure and the classifying map, hence induces a well-defined map
\beqn
\kappa_{n,d}: H^*(B_n K) \to H^{* - 2c_1(d)} (V \qu K).
\eeqn

\begin{lemma}
The maps $\kappa_{n, d}$ induces a map 
\beqn
\kappa_d: H^*(BK) \to H^{*- 2c_1(d)} (V \qu K)
\eeqn
\end{lemma}

\begin{proof}
The map $H^*(B_{n+1} K) \to H^*(B_n K)$ can be constructed Morse-theoretically by extending any Morse--Smale pair on $B_n K$ to $B_{n+1} K$.
\end{proof}

The {\bf quantum Kirwan map} is defined to be the map 
\beqn
\kappa: H^*(BK) \to H^*(V \qu K) \otimes \Lambda,\ \kappa(a) = \sum_{d \in \pi_2^K(X)} q^d \kappa_d(a). 
\eeqn

\subsection{Quantum Kirwan map conjecture: the monotone case}

We prove Theorem \ref{thma}. Recall that one can use the Morse model to define the structure coefficients of the quantum multiplication. In the monotone case, one can use a fixed almost complex structure and consider moduli space of degree $d$ spheres with three marked points attached to gradient rays for three generic Morse--Smale pairs. For the symplectic reduction $V \qu K$, upon choosing the three Morse--Smale pairs, $(f_1, g_1)$, $(f_2, g_2)$ for the two incoming edges and $(f_\infty, g_\infty)$ for the outgoing edge, denote this moduli space by 
\beqn
{\mc M}{}_3^{\rm downstairs}(d; x_1, x_2, x_\infty)
\eeqn
and its compactification by $\ov{\mc M}{}_3^{\rm downstairs}(d; x_1, x_2, x_\infty)$. Define the count by 
\beqn
n_d(x_0, x_1, x_\infty) \in {\mb Z}.
\eeqn
which is zero if the index of the moduli, ${\rm dim} V\qu K - |x_0| - |x_1| + |x_\infty|$, is nonzero. On the homology level these counts induce a map
\beqn
\star_d: H^*(V \qu K; {\mb Z}) \otimes H^*( V \qu K; {\mb Z}) \to H^*(M; {\mb Z}).
\eeqn
Summing over $d$ one obtains the quantum cup product
\beqn
\star:= \sum_d q^d\  \star_d: H^*(V \qu K; \Lambda_{\mb Z}) \otimes H^*( V \qu K; \Lambda_{\mb Z}) \to H^*(V \qu K; \Lambda_{\mb Z}).
\eeqn
In particular, the $d = 0$ part is the classical cup product. Similarly, upon choosing Morse--Smale pairs $(\wh f_i, \wh g_i)$, $i = 1, 2, \infty$ on the classifying space $B_n K$ one has the moduli space (with constant spheres)
\beqn
{\mc M}^{\rm upstairs}_3(\wh x_0, \wh x_2, \wh x_\infty)
\eeqn
for critical points $\wh x_i \in {\rm crit} \wh f_i$.  

Consider scaled curves with two marked points. The moduli space $\ov{\mc M}{}_2^{\rm scaled}$ is homeomorphic to $S^2$ parametrized by the relative position of the two marked points $z_0, z_1 \in {\mb C}$. When $z_0 = z_1$ the configuration is a marked sphere with the two marked points attached to ${\mb C}$; when $z_0 - z_1 = \infty$ the configuration is a marked sphere with two nodes connecting with two copies of ${\mb C}$ with the two marked points $z_0$ and $z_1$ respectively (see Figure \ref{fig:twopoints}).

\begin{figure}[h]
    \centering
    \includegraphics[scale=0.8]{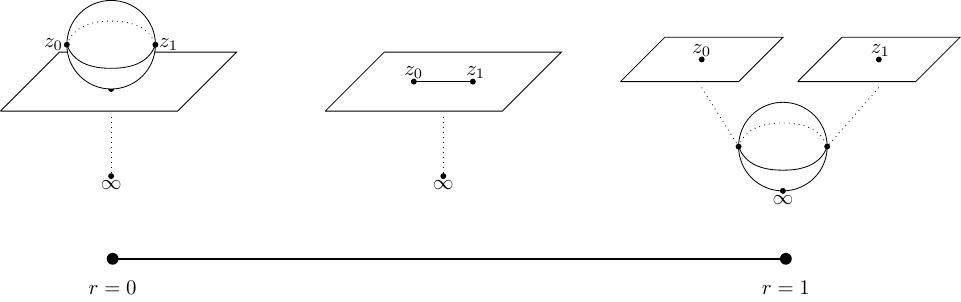}
    \caption{Degenerations of scaled curves with two marked points.}
    \label{fig:twopoints}
\end{figure}

To obtain a 1-parameter family we require that ${\rm Im} z_0 = {\rm Im} z_1$ and ${\rm Re} z_0 \leq {\rm Re} z_1$. Subject to this constraint, the moduli space is denoted by 
\beqn
\ov{\mc M}{}_{2, {\mb R}}^{\rm scaled} \cong [0, 1]
\eeqn
with universal curve $\ov{\mc C}{}_{2, {\mb R}}^{\rm scaled}$. For $r \in [0, 1]$, let $\Sigma_{(r)}$ be the corresponding domain curve. In particular, $\Sigma_{(0)}$ is a complex plane $\Sigma_{(0), \infty}$ union with a sphere $\Sigma_{(0), 01}$ containing the two marked points, and $\Sigma_{(1)}$ is a sphere $\Sigma_{(1),\infty}$ union with two complex planes $\Sigma_{(1), 0}$ and $\Sigma_{(1), 1}$ containing $z_0$ and $z_1$ separately. 

Now we choose the perturbation data, i.e., the family o domain-dependent almost complex structures and various Morse--Smale pairs. We first fix generic Morse--Smale pairs $(\wh f_0, \wh g_0)$, $(\wh f_1, \wh g_1)$, $(\wh f_\infty, \wh g_\infty)$ on $B_n K$ and generic Morse--Smale pairs $(f_0, g_0)$, $(f_1, g_1)$, $(f_\infty, g_\infty)$ on $V \qu K$. We then consider a family $\wh J_{(r)}$ of domain-dependent almost complex structures parametrized by $z \in \Sigma_{(r)}$. We require $\wh J_{(0)}$ and $\wh J_{(1)}$ to take specific values. More precisely, $\wh J_{(0)}$, which induces a domain-dependent almost complex structure $\wh J_{(0), \infty}$ on the plane $\Sigma_{(0), \infty}$, can be used to construct the quantum Kirwan map for the pair $(\wh f_\infty, \wh g_\infty)$ and $(f_\infty, g_\infty)$; $\wh J_{(1)}$, which induces a domain-dependent almost complex structure $\wh J_{(1), 0}$ on $\Sigma_{(1), 0}$ resp. $\wh J_{(1), 1}$ on $\Sigma_{(1), 1}$, can be used to define the quantum Kirwan map for the pair $(\wh f_0, \wh g_0)$ and $(f_0, g_0)$ resp. $(\wh f_1, \wh g_1)$ and $(f_1, g_1)$; moreover, $\wh J_{(1)}$ also induces a domain-dependent almost complex structure on $V\qu K$ parametrized by $z \in \Sigma_{(1), \infty} \cong S^2$, which is required to be a constant so that the quantum cup product can be defined for the three Morse--Smale pairs downstairs. 

For any such family of almost complex structures $\wh J_{(r)}$ and a degree $d$, one obtains a moduli space of stable affine vortices and its cusp compactification, denoted by 
\beqn
\wc{\mc M}{}_{2, [0, 1]}^{\rm vortex}(d).
\eeqn
One can perturb $\wh J_{(r)}$ if necessary (but fix $\wh J_{(0)}$ and $\wh J_{(1)}$) so that every strata is a smooth manifold of the correct dimension. Let 
\beqn
\wc{\mc M}{}_{2, (r)}^{\rm vortex}(d)
\eeqn
be the slice for a specific $r \in [0, 1]$. There are two Poincar\'e bundles
\beqn
\wc{\mc P}_0, \wc{\mc P}_1 \to \wc{\mc M}{}_{2, [0, 1]}^{\rm vortex}(d)
\eeqn
corresponding to the two marked points $z_0$ and $z_1$. One then choose classifying maps
\beqn
\wh \rho_0, \wh \rho_1: \wc{\mc M}{}_{2, [0, 1]}^{\rm vortex}(d) \to B_n K
\eeqn
so that the moduli spaces 
\beqn
\wc{\mc M}{}_{2, [0, 1]}^{\rm vortex}(d; \wh x_0, \wh x_1, x_\infty)
\eeqn
whenever the expected dimension 
\beqn
2c_1(d) - |\wh x_0| - |\wh x_1| + |x_\infty| = 1
\eeqn
satisfy conditions similar to those of Lemma \ref{lemma29}. Notice that there is a natural isomorphism
\beqn
\wc{\mc P}_0|_{\wc{\mc M}{}_{2, (0)}^{\rm vortex}(d)} \cong \wc{\mc P}_1|_{\wc{\mc M}{}_{2, (0)}^{\rm vortex}(d)}
\eeqn
hence we may require 
\beqn
\wh \rho_0|_{\wc{\mc M}{}_{2, (0)}^{\rm vortex}(d)} = \wh \rho_1|_{{\wc{\mc M}{}_{2, (0)}^{\rm vortex}(d)}}.
\eeqn

\begin{lemma}
When $2c_1(d) - |\wh x_0| - |\wh x_1| + |x_\infty| = 1$, the moduli space $\wc {\mc M}{}_{2, [0, 1]}^{\rm vortex}(d; \wh x_0, \wh x_1, x_\infty)$ is a compact 1-dimensional manifold with boundary where boundary points are either once-broken configurations with parameter $r \in (0, 1)$, or unbroken configurations with $r \in \{0, 1\}$. 
\end{lemma}

\begin{proof}
Once-broken configurations are boundary points as one can glue broken Morse trajectories. Points at $r = 0$ slice are boundary points simply because the distance $z_1- z_0$ is a parameter of the vortex equation. The claim that points at $r = 1$ slice are boundary points follows from the gluing result of \cite{Xu_glue}. 
\end{proof}

To identify the contributions of boundary points, one needs to allow the internal edges to acquire length. On the $r = 0$ side of the curve moduli $\ov{\mc M}{}_{2, {\mb R}}^{\rm scaled}$, we allow the edge connecting the sphere and the plane to have positive length up to infinity. Let $r\in [-1, 0]$ parametrize the length of this edge, denoted by $T_{(r)}$, where $r = -1$ corresponds to when the length is infinity (hence breaks). Similarly, on the $r = 1$ side of $\ov{\mc M}{}_{2, {\mb R}}^{\rm scaled}$, we allow the two edges $T_{(r), 0}$ and $T_{(r), 1}$ connecting the planes $\Sigma_{(1), 0}$ and $\Sigma_{(1), 1}$ and the sphere $\Sigma_{(1), \infty}$ to acquire an equal length up to infinity. Let $r \in [1, 2]$ parametrizes this length where $r = 2$ corresponds to when the length is infinity (hence the two edges breaks at the same time). 

We extend $\wc {\mc M}{}_{2, [0,1]}^{\rm vortex}(d; \wh x_0, \wh x_1, x_\infty)$ to $r \in [-1, 0] \cup [1, 2]$ When $r \in [-1, 0]$, on the edge $T_{(r)}$ we include a gradient segment for the pair $(\wh f_\infty, \wh g_\infty)$ in $B_n K$. When $r \in [1, 2]$, on the edge $T_{(r), 0}$ resp. $T_{(r), 1}$, we include a gradient segment for the pair $(f_0, g_0)$ resp. $(f_1, g_1)$. Moreover, extend the almost complex structure, other Morse--Smale pairs, and the classifying maps to the two segments $[-1, 0]$ and $[1, 2]$ in the $r$-independent way. One then obtains extended moduli spaces 
\beqn
\wc {\mc M}{}_{2, [-1, 2]}^{\rm vortex}(d; \wh x_0, \wh x_1, x_\infty).
\eeqn
One may need, if necessary to perturb the Morse functions on the internal edges to achieve transversality for $r \in [-1 + \epsilon, -\epsilon] \cup [1 + \epsilon, 2- \epsilon]$.

It is easy to see that when the virtual dimension is $1$, the extended moduli space is still a compact 1-dimensional manifold with boundary, where boundary points are either configurations with broken external edges for $r \neq -1, 0, 1, 2$, or configurations with internal broken edges for $r = -1$ or $r = 2$. The identification of boundary points is given below. 

\begin{lemma}
When
\beqn
2 c_1(d) - |\wh x_0| - |\wh x_1| + |x_\infty| = 1
\eeqn
the moduli space $\wc {\mc M}{}_{2, [-1, 2]}^{\rm vortex}(d; \wh x_0, \wh x_1, x_\infty)$ is a compact one-dimensional (topological) manifolds with boundary, where boundary points form the disjoint union of 
\beqn
\bigsqcup_{\wh x_\infty \in {\rm crit} \wh f_\infty} {\mc M}_3^{\rm upstairs} (\wh x_0, \wh x_1, \wh x_\infty) \times {\mc M}_1^{\rm vortex}(d; \wh x_\infty, x_\infty),
\eeqn
\beqn
\bigsqcup_{d_0 + d_1 + d_\infty = d} \bigsqcup_{x_0 \in {\rm crit} f_0} \bigsqcup_{x_1 \in {\rm crit} f_1} {\mc M}_1^{\rm vortex}(d_0; \wh x_0, x_0) \times {\mc M}_1^{\rm vortex}(d_1; \wh x_1, x_1) \times {\mc M}_3^{\rm downstairs} (d_\infty; x_0, x_1, x_\infty),
\eeqn
and configurations for $r \in (-1, 0) \cup (0, 1) \cup (1, 2)$ with one external breaking.
\end{lemma}

The following corollary follows easily. 

\begin{cor}
The two chain maps
\beqn
(\wh x_0, \wh x_1) \mapsto \kappa_d( \wh x_0 \cup \wh x_1) 
\eeqn
and 
\beqn
(\wh x_0, \wh x_1) \mapsto \sum_{d_0 + d_1 + d_\infty = d} \kappa_{d_0}(\wh x_0) \star_{d_\infty} \kappa_{d_1}(\wh x_1)
\eeqn
are homotopic. Therefore, for any $a, b \in H^*(BK)$ one has
\beqn
\kappa(a \cup b) = \kappa(a) \star \kappa(b).
\eeqn
\end{cor}
This finishes the proof of Theorem \ref{thma}.

\section{Quantum Steenrod operation}\label{section3}

We review the geometric constructions of the classical and quantum Steenrod operations given by Wilkins \cite{wilkins2020construction} and Seidel--Wilkins \cite{Seidel_Wilkins_2022}.

\subsection{${\mb Z}_p$-equivariant (co)homology}

We give an explicit expression of the classifying space and the universal bundle for the finite group ${\mb Z}_p$. Denote 
\beqn
S^\infty:= \Big\{ w = (w_0, w_1, w_2, \ldots)\ |\ w_i \neq 0\ {\rm for\ only\ finitely\ many\ }i,\ \sum_{i=1}^\infty |w_i|^2 = 1.\Big\}
\eeqn
It is the limit of the sequence of spheres
\beqn
S^{2k-1} = \Big\{ w = (w_0, w_1, w_2, \ldots) \in S^\infty\ |\ w_i = 0 \forall i \geq k \Big\}.
\eeqn
Then $U(1)$ acts on $S^\infty$ preserving each level $S^{2k-1}$. Regarding ${\mb Z}_p\subset U(1)$ as the subgroup of $p$-th roots of unity with the generator $\tau = e^{\frac{2\pi {\bf i}}{p}}$, then 
\beqn
{\mb Z}_p = \{1, \tau, \tau^2, \ldots, \tau^{p-1}\}.
\eeqn
Then we can regard the free quotient 
\beqn
B{\mb Z}_p:= S^\infty/ {\mb Z}_p
\eeqn
as a classifying space of ${\mb Z}_p$ with $E{\mb Z}_p = S^\infty$ the universal bundle. In ${\mb F}_p$ coefficients, the equivariant cohomology of a point is 
\beqn
H_{{\mb Z}_p}^*({\rm pt}; {\mb F}_p) = H^*(B{\mb Z}_p; {\mb F}_p) = \left\{ \begin{array}{ll} {\mb F}_p[t, \theta]/\langle \theta^2 - t \rangle,\ &\ p = 2,\\
                                  {\mb F}_p[t, \theta]/ \langle \theta^2 \rangle,\ &\ p > 2.
                                  \end{array}\right. 
\eeqn
Here $t$ has degree $2$, which can be viewed as the pullback of the universal first Chern class of $BU(1)$ and $\theta$ has degree $1$.

We also specify certain cycles in $B{\mb Z}_p$, following \cite[Section 2]{Seidel_Wilkins_2022}. Define 
\beqn
\Delta_{2k} = \Big\{ w \in S^\infty\ |\ w_k \geq 0,\ w_i = 0\ \forall i \geq k+1 \Big\}
\eeqn
and
\beqn
\Delta_{2k+1} = \Big\{ w \in S^\infty\ |\ e^{- {\bf i} \theta} w_k \geq 0\ {\rm for\ some\ }\theta \in [0, \frac{2\pi}{p}],\ w_i = 0\ \forall i \geq k+1 \Big\}.
\eeqn
One has 
\beq
\partial \Delta_i = \left\{ \begin{array}{ll} \Delta_{i-1} \cup \tau(\Delta_{i-1}) \cup \cdots \cup \tau^{p-1}(\Delta_{i-1}),\ &\ i\ {\rm even},\\  \Delta_{i-1} \cup \tau(\Delta_{i-1}),\ &\ i\ {\rm odd}         \end{array}\right.
\eeq

After the ${\mb Z}_p$-action, the chains $\Delta_i$ become cycles in $B{\mb Z}_p$ in ${\mb F}_p$-coefficients and their homology classes form a basis of $H_*^{{\mb Z}_p}({\rm pt}; {\mb F}_p)$. By abuse of notations, we still denote by the same symbol the corresponding cycle/homology class, i.e.
\beqn
\Delta_i \in H_i (B{\mb Z}_p; {\mb F}_p).
\eeqn

\begin{lemma}\cite[Lemma 2.1]{Seidel_Wilkins_2022}\label{lemma31}
Let 
\beqn
\delta_*: H_*(B{\mb Z}_p; {\mb F}_p) \to H_*(B{\mb Z}_p; {\mb F}_p) \otimes H_*(B{\mb Z}_p; {\mb F}_p)
\eeqn
be the map induced by the diagonal embedding of $B{\mb Z }_p$. Then 
\beqn
\delta_* \Delta_i = \left\{ \begin{array}{rl} \displaystyle \sum_j \Delta_j \otimes \Delta_{i-j},\ &\ i\ {\rm is\ odd\ or\ } p = 2,\\[0.5cm]
                         \displaystyle  \sum_{j\ {\rm even}}  \Delta_j \otimes \Delta_{i-j},\ &\ i\ {\rm is\ even\ and\ }p > 2.
                           \end{array}\right.
\eeqn
    
\end{lemma}

\subsection{Quantum Steenrod operations following Seidel--Wilkins}

We recall the construction of Seidel--Wilkins \cite{Seidel_Wilkins_2022} of the quantum Steenrod operation of a compact monotone symplectic manifold, which in particular applies to the symplectic reduction $V \qu K$. It is a collection of maps
\beqn
Q\Sigma_b: H^*(V \qu K; {\mb F}_p) \to H^*(V \qu K; \Lambda_{{\mb F}_p})[t, \theta]
\eeqn
(for characteristic $p>2$) parametrized by $b \in H^*(V \qu K; {\mb F}_p)$ which extends Wilkins' construction of the quantum Steenrod squares (\cite{wilkins2020construction}). The minor change we will make is that we perturb the almost complex structure and do not use inhomogeneous term for the Cauchy--Riemann equation. The possible loss of transversality for constant sphere is restored by perturbing the Morse functions, the same way as in \cite{wilkins2020construction}.

Fix a Morse--Smale pair $(f, g)$ on $V \qu K$. Consider a sphere $S^2 = {\mb C} \cup \{\infty\}$ with $p+1$ incoming marked points $z_0 = 0$ and $z_1, \ldots, z_p$ being the $p$-th roots of unity and one output placed at $\infty \in S^2$, also denoted by $z_\infty$. Then ${\mb Z}_p$ acts on the marked sphere by permuting the marked points in an obvious way, fixing $z_0$ and $z_\infty$. We denote by $\tau^j(i)$ for $\tau^j \in {\mb Z}_p$ and $i \in \{0, 1, \ldots, p, \infty\}$ the induced action on the index set. Fix a $p+2$-tuple of families of perturbations of the Morse function $f$, denoted by $f_{i, w}: V \qu K \to {\mb R}$ with $i = 0, \ldots, p, \infty$ and $w \in S^\infty$ satisfying 
\beqn
f_{\tau^j (i), \tau^j(w)} = f_{i, w},\ \forall \tau^j \in {\mb Z}_p,\ i \in \{0, 1, \ldots, p, \infty\}.
\eeqn
Assume these perturbations are supported away from critical points of $f$. Also choose a ${\mb Z}_p$-equivariant family of almost complex structures $J_w(z)$ on $V \qu K$ parametrized by $z \in S^2$ and $w \in S^\infty$ satisfying
\beqn
J_{\tau^j (w)}(\tau^j (z)) = J_w (z),\ \forall \tau^j \in {\mb Z}_p.
\eeqn
We require that near marked points, $J_w (z)$ is a fixed almost complex structure. 

Consider moduli spaces of ``treed holomorphic spheres.'' For each spherical class $d \in \pi_2(X)$, a $p+2$-tuple of critical points $(x_0|x_1, \ldots, x_p|x_\infty)$, and a subset $\Delta \subset S^\infty$, consider the moduli space 
\beqn
 {\mc M}_{p+2, \Delta}^{\rm downstairs} (d)
\eeqn
consisting of tuples 
\beqn
(w, u, \vec{y}):= (w, u, y_0|y_1, \ldots, y_p|y_\infty)
\eeqn
where $w \in \Delta$, $u: S^2 \to V \qu K$, and $y_i: I_i \to V \qu K$ ($I_i = (-\infty, 0]$ if $i \neq \infty$ and $I_\infty = [0, +\infty)$) satisfying \footnote{We may also allow $f_{i, w}$ to be domain-dependent, i.e., varying with $t \in I_i$ in a compact subset. }
\begin{align*}
&\ \ov\partial_{J_w} u = 0,\ &\ y_i'(t) + \nabla f_{i, w}(y_i (t)) = 0,
\end{align*} 
the matching conditions
\beqn
y_i(0) = u(z_i),\ i = 0, \ldots, p, \infty,
\eeqn
and 
\beqn
u_*[S^2] = d \in \pi_2( V \qu K).
\eeqn
There is a free ${\mb Z}_p$-action on ${\mc M}_{p+2, S^\infty}^{\rm downstairs} (d)$ given by 
\beqn
\tau^j (w, u, y_i) = (\tau^j (w), u \circ \tau^j, y_{\tau^j(i)} ).
\eeqn

Let $x_0, \ldots, x_p, x_\infty$ be critical points of $f$. Denote 
\beqn
{\mc M}_{p+2, \Delta}^{\rm downstairs} (d; x_0|x_1, \ldots, x_p|x_\infty) \subset {\mc M}_{p+2, S^\infty}^{\rm downstairs} (d)
\eeqn
be the subset of elements $(w, u, y_i)$ with $w \in \Delta$ such that $y_i$ converges to $x_i$ at infinity. Then for $\tau^j \in {\mb Z}_p$, let 
\beqn
(x_1^{(j)}, \ldots, x_p^{(j)})
\eeqn
be the $p$-tuple obtained by cyclically permuting $(x_1, \ldots, x_p)$ for $j$ times to the right (so $x_1^{(1)} = x_p$). Then by the symmetry of the equation, one has 
\beq\label{moduli_symmetry}
\tau^j \Big( {\mc M}_{p+2, \Delta}^{\rm downstairs} (d; x_0|x_1, \ldots, x_p| x_\infty) \Big) = {\mc M}_{p+2, \tau^j (\Delta)}^{\rm downstairs} (d; x_0|x_1^{(j)}, \ldots, x_p^{(j)}|x_\infty).
\eeq


\begin{lemma}\cite[Lemma 4.1]{Seidel_Wilkins_2022}\label{lemma32}
Choose a generic family of almost complex structures $J_w$.
\begin{enumerate}

\item For each $i$, ${\mc M}_{p+2, \Delta_i}^{\rm downstairs} (d; x_0|x_1, \ldots, x_p|x_\infty)$ is a smooth manifold with boundary and 
\beqn
{\rm dim} {\mc M}_{p+2, \Delta_i}^{\rm downstairs} (d; x_0|x_1, \ldots, x_p| x_\infty) = i + 2c_1(d) + |x_\infty| - |x_0| - \cdots - |x_p|.
\eeqn
The boundary points are tuples $(w, u, y_i)$ with $w \in \partial \Delta_i$.

\item When the dimension is zero, the moduli space ${\mc M}_{p+2, \Delta_i}^{\rm downstairs} (d; x_0|x_1,\ldots, x_p|x_\infty)$ is a finite set of tuples $(w, u, y_0|y_1, \ldots, y_p|y_\infty)$ with $w \in {\rm Int} \Delta_i$.

\item When the dimension is one, ${\mc M}_{p+2, \Delta_i}^{\rm downstairs} (d; x_0|x_1, \ldots, x_p|x_\infty)$ can be compactified to a 1-dimensional manifold with boundary, where boundary points corresponding to configurations which has either exactly one broken edge with $w \in {\rm Int}\Delta_i$, or configurations $(w, u, y_i)$ with $w \in \partial \Delta_i$. The latter, by the identification \eqref{moduli_symmetry}, can be identified with the disjoint union 
\beqn
\bigsqcup_{j} {\mc M}_{p+2, \Delta_{i-1}}^{\rm downstairs} (d; x_0|x_1^{(j)}, \ldots, x_p^{(j)} |x_\infty)\ {\rm over}\ \left\{ \begin{array}{ll} j = 0, \ldots, p-1,\ &\ i \in 2 {\mb Z},\\
 j = 0, 1,\ &\ i \notin 2{\mb Z}.\end{array}\right.
\eeqn
\end{enumerate}
\end{lemma}

Using the orientation on the moduli space of holomorphic spheres and the orientations on the (un)stable manifolds of the Morse--Smale pair, one can define the counts of the above moduli spaces, which is nonzero only when the dimension of the moduli space is zero. Define\footnote{We can allow the Morse--Smale pairs at $z_0$ or $z_\infty$ to be different from the pair at $z_1, \ldots, z_p$.}
\beqn
\Sigma_{i, d}: CM^*(f)^{\otimes p+1} \to CM^{*- i - 2 c_1(d)} (f)
\eeqn
by spanning 
\beqn
\Sigma_{i,d}(x_0\otimes x_1\otimes \cdots \otimes x_p):= \sum_{x_\infty \in {\rm crit} f} \# {\mc M}_{p+2, \Delta_i}^{\rm downstairs} (d; x_0|x_1, \ldots, x_p|x_\infty) x_\infty. 
\eeqn
Restrict to $x_1 = \cdots = x_p = b$, one defines $\Sigma_{d, b}: CM^*(f) \to CM^*(f)[t, \theta]$ by 
\beqn
x\mapsto (-1)^{|b||x|} \sum_{k\geq 0} \left( \Sigma_{2k, d}(x \otimes b\otimes \cdots \otimes b) + (-1)^{|b| + |x|} \Sigma_{2k+1, d} (x\otimes b \otimes \cdots \otimes b) \theta \right) t^k
\eeqn
Then he quantum Steenrod operation is the sum over the degree
\beqn
Q\Sigma_b:= \sum_d q^d \Sigma_{d, b}: CM^*(f)\otimes \Lambda \to CM^*(f; \Lambda)[t, \theta].
\eeqn
By \cite[Lemma 4.4]{Seidel_Wilkins_2022}, when $b$ is a cocycle, $Q\Sigma_b$ is a chain map and the induced map on cohomology only depends on the cohomology class of $b$.

\subsection{Classical Steenrod operation on classifying space}

The construction of Seidel--Wilkins sketched above extended the original Morse-theoretic construction of the classical Steenrod operations by Betz--Cohen \cite{Betz_Cohen_1994}. Indeed, if we replace the symplectic manifold $V \qu K$ by any compact manifold $M$, while requiring, in the definition of the moduli spaces the holomorphic spheres to be constant maps from $S^2$ to $M$ (hence no dependence on $d$), then one can produce a map 
\beqn
\Sigma_b: CM^*(f) \to CM^*(f)[t, \theta]
\eeqn
which descends to a well-defined map on cohomology that coincides with the classical Steenrod operation.   

We can slightly generalize the Morse-theoretic construction to the classifying space. Choose a concrete model of $BK$ which admits a manifold approximation
\beqn
B_n K \subset B_{n+1} K \subset BK.
\eeqn
Here $B_n K$ are smooth oriented manifolds with strictly increasing dimensions. Then we have 
\beqn
\varprojlim_{n \to \infty} H^*(B_n K) = H^*(BK).
\eeqn
For any $b\in H^*(BK)$, let $b_n$ be its restriction to $B_n K$. By the naturality of the classical Steenrod operations, one has the following commutative diagram
\beqn
\xymatrix{ H^*(BK) \ar[r] \ar[d]_{\Sigma_b} & H^*(B_{n+1} K) \ar[r] \ar[d]_{\Sigma_{b_{n+1}}} & H^*(B_n K)\ar[d]_{\Sigma_{b_n}} \\
 H^*(BK)[t, \theta] \ar[r] & H^*(B_{n+1} K)[t, \theta] \ar[r] & H^*(B_n K)[t, \theta] }.
\eeqn
Then for fixed $a, b \in H^*(BK)$, $\Sigma_b(a)$ is determined by $\Sigma_{b_n}(a_n)$ for any sufficiently large $n$. As $B_n K$ are finite-dimensional compact manifolds, the Steenrod operations on $BK$ can be realized by the Morse-theoretic construction on $B_n K$ for large $n$.

\section{The equivariant quantum Kirwan map}\label{section4}

The equivariant quantum Kirwan map is possible because the domain admits a natural rotational symmetry. We sketch the construction first. Consider the moduli space of stable affine vortices with one interior marking, regarded as the origin of ${\mb C}$. Then ${\mb Z}_p$ acts on ${\mb C}$ fixing the origin $z_0 = 0$. Let $\wh J_w (z)$ be a family of $K$-invariant $\omega_V$-compatible almost complex structures on ${\mb C}$ depending on $w \in S^\infty$ and $z \in {\mb C}$ such that 
\beqn
\wh J_{\tau^j(w)}(\tau^j(z)) = \wh J_w (z)\ \forall \tau^j \in {\mb Z}_p.
\eeqn
For any subset $\Delta \subset S^\infty$, consider the moduli space 
\beqn
{\mc M}_{1, \Delta}^{\rm vortex}(d)
\eeqn
of tuples $(w, A, u)$ where $w \in \Delta$ and $(A, u)$ is a degree $d$ affine vortex with respect to the complex structure $\wh J_w$. It also admits the cusp compactification 
\beqn
\wc {\mc M}_{1, \Delta}^{\rm vortex}(d)
\eeqn
and a Poincar\'e bundle 
\beqn
\wc {\mc P}_{0, \Delta} \to \wc{\mc M}_{1, \Delta}^{\rm vortex}(d).
\eeqn
Then by the equivariance of the almost complex structure $\wh J_w (z)$, there is a free ${\mb Z}_p$-action on $\wc{\mc M}_{1, S^\infty}^{\rm vortex}(d)$ with natural lifts on the Poincar\'e bundle $\wc {\mc P}_{0, S^\infty}$, such that 
\beqn
\tau^j( \wc{\mc M}_{1, \Delta}^{\rm vortex}(d)) = \wc{\mc M}_{1, \tau^j(\Delta)}^{\rm vortex}(d), \forall \Delta \subset S^\infty.
\eeqn
For each $i$, choose ${\mb Z}_p$-invariant classifying maps
\beqn
\wh \rho_0: \wc {\mc M}{}_{1, {\mb Z}_p(\Delta_i)}^{\rm vortex}(d) \to B_n K
\eeqn
for a sufficiently large $n$. 

We still use the Morse model. Fix a Morse--Smale pair $(\wh f_0, \wh g_0)$ on $B_n K$ and a Morse--Smale pair $(f_\infty, g_\infty)$ on $V \qu K$. Consider zero or one-dimensional moduli spaces 
\beqn
\wc {\mc M}{}_{1,\Delta_i}^{\rm vortex}(d; \wh x_0, x_\infty) = \left\{ (w, u, \wh y_0, y_\infty)\ \left|\ \begin{array}{c}  (w, u) \in \wc {\mc M}{}_{1, \Delta_i}^{\rm vortex}(d),\ \wh y_0: I_0 \to B_n K,\ y_\infty: I_\infty \to V \qu K,\\
\wh y_0' + \nabla \wh f_{w, 0} (\wh y_0) = 0,\ \wh y_0(\infty) = \wh x_0,\ \wh y_0(0) = \wh \rho_0 (w, u),\\
y_\infty' + \nabla f_\infty(y_0) = 0,\ y_\infty(\infty) = x_\infty,\  \ev_\infty(u) = y_\infty(0).
\end{array} \right. \right\}.
\eeqn
As we did in Section \ref{section2}, this notation allows $\wh y_0$ or $y_\infty$ to have breakings. 

\begin{lemma}\label{lemma41}
Suppose the virtual dimension $i + 2c_1(d) - |\wh x_0| + |x_\infty|$ is at most one, then one can choose a classifying map $\wh \rho_0$ such that
\begin{enumerate}

\item If the virtual dimension is negative, the above moduli space is empty.

\item If the virtual dimension is zero, then for any configuration $(w, u, \wh y_0, y_\infty)$ in the moduli space, $w \in {\rm Int} \Delta_i$ and $\wh y_0, y_\infty$ have no breakings.

\item If the virtual dimension is one, then the moduli space is a compact one-dimensional manifold with boundary, where boundary configurations consists of either $w \in {\rm Int}\Delta_i$ with $\wh y_0$ or $y_\infty$ has exactly one breakings, or $w \in \partial \Delta_i$ with $\wh y_0$ and $y_\infty$ having no breaking.
\end{enumerate}
\end{lemma}

\begin{proof}
Similar to the case of Lemma \ref{lemma29}.
\end{proof}

Then consider the count of zero-dimensional moduli spaces, denoted by 
\beqn
n_{i,d}^{eq} (\wh x_0, x_\infty) \in {\mb F}_p.
\eeqn
Define
\beqn
\kappa_d^{eq} (\wh x_0) = \sum_{x_\infty \in {\rm crit} f_\infty} \sum_{k=0}^\infty \left( n_{2k, d}^{eq}(\wh x_0, x_\infty) + \theta n_{2k+1, d}^{eq}(\wh x_0, x_\infty) \right) t^k.
\eeqn

\begin{cor}
The map $\kappa_d^{eq}: CM^*(\wh f_0; {\mb F}_p) \to CM^*(f_\infty; {\mb F}_p)[t, \theta]$ is a chain map which is independent of the choices up to homotopy. Moreover, the induced map on equivariant cohomology does not depend on the choice of large $n$. 
\end{cor}

\begin{proof}
The extra boundary configurations described in (3) of Lemma \ref{lemma41} corresponding to $w \in \partial \Delta_i$ contribute zero to the count (in characteristic $p$). 
\end{proof}

Hence $\kappa_d^{eq}$ induces a map on cohomology. Define the ${\mb Z}_p$-equivariant quantum Kirwan map by 
\beqn
\kappa^{eq}: H^*(BK; {\mb F}_p) \to H^*(V \qu K; {\mb F}_p)[t, \theta]
\eeqn
by 
\beqn
\kappa^{eq} = \sum_{d \in \pi_2^K(V)} q^d \kappa_d^{eq}.
\eeqn

\subsection{Special values}

Now we prove that the classical part of the equivariant quantum Kirwan map (by specializing to $q = 0$) coincides with the classical Kirwan map with no equivariant parameters.

\begin{prop}\label{prop43}
For any $a \in H_K^*(X; {\mb F}_p)$, one has 
\beqn
\kappa^{\it eq}(a)|_{q = 0} = \kappa(a)|_{q = 0} \in H^*(V \qu K; {\mb F}_p).
\eeqn
\end{prop}

\begin{proof}
The $q = 0$ part of $\kappa^{\it eq}$ only has contributions from trivial affine vortices. For any fixed almost complex structure $\wh J$ (independent of $w \in S^\infty$), the moduli space ${\mc M}_1^{\rm vortex}(d)$ with $d = d_0 = 0$ is compact and homeomorphic to $V \qu K$. Therefore, 
\beqn
\wc{\mc M}_{1, \Delta_i}^{\rm vortex}(d_0) \cong \Delta_i \times V \qu K.
\eeqn
The Poincar\'e bundle is the pullback of the bundle $\mu^{-1}(0) \to V \qu K$. One can choose a classifying map independent of the factor $\Delta_i$ such that the moduli spaces
\beqn
\wc {\mc M}{}_1^{\rm vortex}(d_0; \wh x_0, x_\infty)
\eeqn
satisfy conditions of Lemma \ref{lemma29} whenever $2c_1(d_0) - |\wh x_0| + |x_\infty| \leq 1$. Therefore, using the pullback of this classifying map to $\wc{\mc M}{}_{1, \Delta_i}^{\rm vortex}(d_0)$, when $i + 2c_1(d_0) - |\wh x_0| + |x_\infty| = 0$ and $i \geq 1$, one has
\beqn
\wc {\mc M}{}_{1, \Delta_i}^{\rm vortex}(d_0; \wh x_0, x_\infty) = \emptyset.
\eeqn
This implies that 
\beqn
n_{i, d_0}^{eq}(\wh x_0, x_\infty) = 0\ \forall i \geq 1. \qedhere
\eeqn
\end{proof}

\begin{prop}
Under the monotonicity condition, $\kappa^{eq}(1) = 1$.
\end{prop}

\begin{proof}
Assume $\wh x_0$ is a global maximum of $\wh f_\infty$ which represents the unit element of $H^*(BK)$. By the dimension formula, $n_{i, d}^{eq}(\wh x_0, x_\infty) \neq 0$ only when 
\beqn
i + 2c_1(d) - |\wh x_0| + |x_\infty| = 0 \Longrightarrow |x_\infty| = - i - 2c_1(d).
\eeqn
The monotonicity condition implies that $|x_\infty| \geq 0$ only when $i = c_1(d) = 0$. Therefore, the only contribution comes from the classical part, which gives $\kappa^{eq}(1) = 1$.
\end{proof}

\section{Proof of Theorem \ref{thmb}}\label{section5}

The proof of Theorem \ref{thmb} is very similar to that of Theorem \ref{thma}. We only need to use a different moduli space and keep track of the additional parameter $w \in S^\infty$. The argument is also analogous to that of \cite[Proposition 4.8]{Seidel_Wilkins_2022} and we will be sketchy. 

We first describe the moduli spaces. Consider the complex plain ${\mb C}$ together with marked points 
\beqn
z_0 = 0\ {\rm and}\  z_i = a \xi_i, i = 1, \ldots, p
\eeqn
where $\xi_1, \ldots, \xi_p \in S^1$ are the $p$-th roots of unity. Such marked scaled curves form a 1-dimensional moduli space parametrized by $a \in (0, \infty)$, which can be compactified by adding two boundary points. When $a \to 0$, a sphere with $p$ marked points bubble off; when $a \to \infty$, $p$ complex planes with one marking are connected by one sphere. In both cases, the marked points or nodes on the spheres can be regarded as placed at the positions of all $p$-th roots of unity (see Figure \ref{fig:ppoints}).

\begin{figure}[h]
    \centering
    \includegraphics[scale = 0.9]{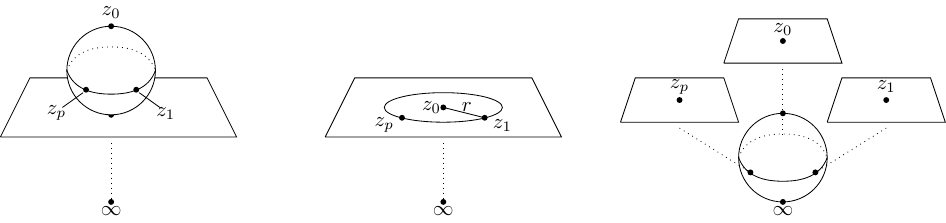}
    \caption{The one-dimensional moduli space of marked scaled curves with $p$ marked points.}
    \label{fig:ppoints}
\end{figure}

Let $\ov{\mc M}{}_{p+1, {\mb R}}^{\rm scaled}$ denote the moduli space of these marked scaled curves and $\ov{\mc C}{}_{p+1, {\mb R}}^{\rm scaled}$ the universal curve. There is a ${\mb Z}_p$-action on the family $\ov{\mc C}{}_{p+1, {\mb R}}^{\rm scaled} \to \ov{\mc M}{}_{p+1, {\mb R}}^{\rm scaled}$ induced by the cyclic symmetry. The moduli space $\ov{\mc M}{}_{p+1, {\mb R}}^{\rm scaled}$ can be identified with an interval $[0, 1]$ with the parameter denoted by $r$. 

Consider a family of almost complex structures $\wh J_{(r), w}(z)$ parametrized by $r$, $w \in S^\infty$, and $z \in \ov{\mc C}{}_{p+1, {\mb R}}^{\rm scaled}$, with the equivariance condition:
\beqn
\wh J_{(r), \tau^j(w)}(\tau^j z) = \wh J_{(r), w}(z)\ \forall \tau^j \in {\mb Z}_p.
\eeqn
Then one can consider the moduli space 
\beqn
\wc{\mc M}{}_{p+1, [0, 1], \Delta}^{\rm vortex}(d)
\eeqn
of tuples $(r, w, u)$ where $r \in [0, 1]$, $w \in \Delta \subset S^\infty$, and $u$ is a stable affine vortex in the cusp compactification. There are Poincar\'e bundles
\beqn
\wc {\mc P}_{i, [0, 1],\Delta} \to \wc {\mc M}{}_{p+1, [0, 1], \Delta}^{\rm vortex}(d), \ i = 0, 1, \ldots, p.
\eeqn

We choose $(r, w)$-dependent Morse--Smale pairs $(\wh f_{(r), j, w}, \wh g_{(r), j, w})$ on $B_n K$, $j = 1, \ldots, p$ such that 
\beqn
(\wh f_{(r), i, w}, \wh g_{(r), i, w}) = (\wh f_{(r), \tau^j(i), \tau^j(w)}, \wh g_{(r), \tau^j (i), \tau^j(w)}),\ i = 1, \ldots, p,\ \tau^j \in {\mb Z}_p,
\eeqn
all of which are perturbations (away from critical points) of a fixed Morse--Smale pair $(\wh f, \wh g)$ on  $B_n K$. Choose an $(r, w)$-independent Morse--Smale pairs $(\wh f_0, \wh g_0)$ on $B_n K$ and an $(r, w)$-independent Morse--Smale pairs $(f_\infty, g_\infty)$ on $V\qu K$.


One can perturb the almost complex structure $\wh J_{(r), w}$ so that on each stratum of the above cusp compactification, the evaluation map at infinity is transverse to all stable manifolds of $f_\infty$. Then we choose suitable classifying maps
\beqn
\wh \rho_i: \wc {\mc M}{}_{p+1, [0, 1], {\mb Z}_p(\Delta_i)}^{\rm vortex}(d) \to B_n K,\ i = 0, \ldots, p
\eeqn
satisfying a natural ${\mb Z}_p$-equivariance condition. In addition, notice that on the $r= 0$ slice of the moduli space the Poincar\'e bundles are all naturally isomorphic. Hence we can require that the $p+1$ classifying maps are identical on the $r = 0$ slice. Then define 
\beqn
\wc{\mc M}{}_{p+1, [0,1], \Delta_i}^{\rm vortex}(d; \wh x_0| \wh x_1, \ldots, \wh x_p| x_\infty)
\eeqn
by adding gradient rays at incoming and outgoing edges. The classifying maps can be chosen so that when the expected dimension is at most $1$, moduli spaces
\beqn
\wc {\mc M}{}_{p+1, [0, 1], \Delta_i}^{\rm vortex}(d; \wh x_0|\wh x_1, \ldots, \wh x_{p}|x_\infty)
\eeqn
consist of points with no further degenerations other than shapes described in Figure \ref{fig:ppoints} or once-broken trajectories. 

By using the gluing result of \cite{Xu_glue}, one can show that the one-dimensional moduli spaces are actually 1-dimensional compact manifolds with boundary. Next we need to identify the counts of boundary points. Similar to the proof of Theorem \ref{thma}, we extend the parameter $r$ to $[-1, 2]$ and inserting gradient segments at interior nodes. The details are given below.

For $r \in [-1, 0]$, we insert an edge $[-l(r), l(r)]$ between the sphere and the plane with $l(0) = 0$ and $l(r) \to +\infty$ when $r \to -1$, while adding a gradient segment in $B_n K$ of an $(r, w)$-independent Morse--Smale pair $(\wh f_\infty, \wh g_\infty)$ defined on $[-l(r), l(r)]$ with the obvious matching condition. We choose the pair such that the classical Steenrod operation on $B_n K$ is defined for the pairs $(\wh f_{(0), i, w}, \wh g_{(0), i, w})$ for $i = 0, \ldots p$ and $(\wh f_\infty, \wh g_\infty)$. For this range of $r$, the almost complex structure, the other Morse--Smale pairs, and the classifying maps remain constant. 

For $r\in [1, 2]$, we insert $p+1$ copies of $[-l(r), l(r)]$ to each of the node connecting the plane $\Sigma_{(1), i}$, $i = 0, \ldots, p$, and the sphere $\Sigma_{(1), \infty}$ with $l(1) = 0$ and $l(r) \to +\infty$ when $r \to 2$, while adding gradient segments of Morse--Smale pairs $(f_{(r), i, w}, g_{(r), i, w})$ on $V \qu K$ on the $i$-th edge. We may take all of them to be perturbations (away from critical points) of the same Morse--Smale pair $(f, g)$ on $V \qu K$. Here we need to take a further perturbation, requiring that $(f_{(r), i, w}, g_{(r), i, w})$ to be domain-dependent, i.e., depending on points on the interval $[-l(r), l(r)]$. Moreover, when $r = 2$ and each of the edge breaks into two semi-infinite edges $[0, +\infty) \cup (-\infty, 0]$, we require that, for $i = 1, \ldots, p$, the restriction of $(f_{(2), i, w}, g_{(r), i, w})$ to $[0, +\infty)$ is $w$-independent (as it is supposed to be part of the data for the nonequivariant quantum Kirwan map). The equivariance requirement then forces that the restrictions are also $i$-independent. Meanwhile, their restrictions to $(-\infty, 0]$ could be $w$-dependent. We also need to modify the incoming Morse--Smale pairs $(\wh f_{(r), i, w}, g_{(r), i, w})$ for $r\in [1, 2]$ such that, when $r = 2$, for $i = 1, \ldots, p$, $(f_{(2), i, w}, g_{(2), i, w})$ is $w$-independent (and hence $i$-independent).

With all the data specified above, we obtain an extension of the moduli space, denoted by 
\beqn
\wc{\mc M}{}_{p+1, [-1, 2], \Delta}^{\rm vortex}(d; \wh x_0|\wh x_1, \ldots, \wh x_p|x_\infty)
\eeqn
which is again a compact 1-dimensional manifolds with boundary. The boundary points within $(-1, 2)$ are all configurations with exactly one external breakings. We need to identify the boundary contributions at $r = -1$ and $r = 2$.

The contribution at $r = -1$ is a chain map 
\beqn
{\mc L}_{(-1)}: CM^*( \wh f_0) \otimes CM^*(\wh f)^{\otimes p} \to CM^*(f_\infty).
\eeqn
The underlying moduli space is the fibre product
\beqn
\bigsqcup_{\wh x_\infty \in {\rm crit} \wh f_\infty} {\mc M}{}^{\rm upstairs}_{p+1, \Delta_i}(\wh x_0|\wh x_1, \ldots, \wh x_p|\wh x_\infty) \times_{S^\infty} \wc {\mc M}{}_{1, \Delta_i}^{\rm vortex}( \wh x_\infty, x_\infty).
\eeqn
Following the same argument as \cite[Proposition 4.8]{Seidel_Wilkins_2022}, we may regard the moduli space as depending on two parameters $(w_1, w_2)$ living in a specific subset of $S^\infty \times S^\infty$. The equivariance condition on the perturbation data implies that the induced counts only depend on the cycle in $B{\mb Z}_p \times B{\mb Z}_p$. Denote the map corresponding to a cycle $[S]$ by ${\mc L}_{(-1), [S]}$. Then the $r = -1$ component of the above moduli space gives the map ${\mc L}_{(-1), \delta([\Delta_i])}$, corresponding to the diagonal cycle of $[\Delta_i]$. One also know that homologous cycles give homotopic maps. By Lemma \ref{lemma31}, we know up to homotopy
\beqn
{\mc L}_{(-1)} \simeq \left\{\begin{array}{ll}  \displaystyle \sum_{i_1 + i_2 = i} {\mc L}_{(-1), [\Delta_{i_1}]\times [\Delta_{i_2}]},\ &\ i\ {\rm odd}\ {\rm or}\ p = 2,\\[0.2cm]
\displaystyle \sum_{i_1 + i_2 = i,\atop i_1, i_2\ {\rm even}} {\mc L}_{(-1), [\Delta_{i_1}]\times [\Delta_{i_2}]},\ &\ i\ {\rm even}\ {\rm and}\ p > 2.
\end{array} \right.
\eeqn
Therefore, the map on homology induced from ${\mc L}_{(-1)}$ is equal to 
\beqn
\kappa^{eq} \circ \Sigma_b: H^*(BK; {\mb F}_p) \to H^*(V \qu K; \Lambda_{{\mb F}_p})[t, \theta].
\eeqn

Similarly, denote the contribution at $r = 2$ by 
\beqn
{\mc L}_{(2)}: CM^*(\wh f_0) \otimes CM^*(\wh f)^{\otimes p} \to CM^*( f_\infty).
\eeqn
The underlying moduli space is the fibre product
\beqn
\bigsqcup_{d_0 + \cdots + d_p + d_\infty = d \atop x_0, \ldots, x_p \in {\rm crit} f} \wc {\mc M}{}_{1, \Delta_i}^{\rm vortex}(d_0; \wh x_0, x_0) \times_{S^\infty} {\mc M}{}_{p+2, \Delta_i}^{\rm downstairs}(d_\infty; x_0|x_1, \ldots, x_p|x_\infty) \times \prod_{i=1}^p \wc{\mc M}{}_1^{\rm vortex}(d_i; \wh x_i, x_i).
\eeqn
Using Lemma \ref{lemma31} again, we can identify, up to homotopy, the contribution of boundary points on the $r = 2$ side is a chain map whose restriction to $\wh x_1 = \cdots \wh x_p = b$ is
\beqn
\wh x_0 \mapsto Q\Sigma_{\kappa(b)} (\kappa^{eq}(\wh x_0)).
\eeqn

This finishes the proof of Theorem \ref{thmb}.

\bibliography{reference}

\bibliographystyle{amsalpha}

\end{document}